\documentclass[12pt]{article}

\usepackage{pslatex}
\usepackage{fancyhdr}
\usepackage{graphicx}
\usepackage{geometry}

\RequirePackage{amsfonts,amssymb,amsmath,amscd,amsthm}
\RequirePackage{txfonts}
\RequirePackage{graphicx}
\RequirePackage{xcolor}
\RequirePackage{geometry}
\RequirePackage{enumerate}

\def\figurename{Figure} % Replace the colon that normally appears after the Figure number by a period.
\makeatletter
\renewcommand{\fnum@figure}[1]{\figurename~\thefigure.}
\makeatother

\def\tablename{Table} % Replace the colon that normally appears after the Figure number by a period.
\makeatletter
\renewcommand{\fnum@table}[1]{\tablename~\thetable.}
\makeatother

\usepackage{amsmath}
\usepackage{amssymb}
\usepackage{amsfonts}
\usepackage{amsthm,amscd}

\newtheorem{theorem}{Theorem}[section]
\newtheorem{lemma}[theorem]{Lemma}
\newtheorem{corollary}[theorem]{Corollary}
\newtheorem{proposition}[theorem]{Proposition}

\theoremstyle{example}
\newtheorem{example}[theorem]{Example}

\theoremstyle{definition}
\newtheorem{definition}[theorem]{Definition}

\theoremstyle{remark}
\newtheorem{remark}[theorem]{Remark}

\numberwithin{equation}{section}

\def\cal{\mathcal}

\usepackage{authblk}

\title{Hom-left-symmetric color dialgebras, Hom-tridendriform color algebras and  Yau's twisting generalizations}

\author[1]{Ibrahima Bakayoko}
\author[2]{Sergei Silvestrov}
\affil[1]{D\'epartement de Math\'ematiques, Universit\'e de N'Z\'er\'ekor\'e, BP 50 N'Z\'er\'ekor\'e, Guin\'ee, \text{ibrahimabakayoko27@gmail.com}}
\affil[2]{Division of Applied Mathematics, School of Education, Culture and Communication, M\"{a}lardalen University, Box 883, 72123 V{\"a}ster{\aa}s, Sweden,
\text{sergei.silvestrov@mdh.se}}

\date{}

\begin{document}
\maketitle

% % % % % % \thispagestyle{empty} \setcounter{page}{1}
% % % % % % % ------- [First Page Running Head] - place it immediately after title! ------
% % % % % % \thispagestyle{fancy} \fancyhead{}
% % % % % % \fancyhead[L]{{\LARGE A}frican {\LARGE D}iaspora {\LARGE J}ournal of {\LARGE M}athematics\\
% % % % % % Volume X, Number X, pp. {\thepage--\pageref{lastpage-01} (2013)}} % put \label{lastpage-xx} on the last page!
% % % % % % \fancyhead[R]{ISSN 1539-854X  \\ {\tt{www.math-res-pub.org/adjm}}}
% % % % % % \fancyfoot{}
% % % % % % \renewcommand{\headrulewidth}{0pt}
%------------------------------------------------------------------------------

\begin{abstract}
The goal of this paper is to introduce and give some constructions and study properties of Hom-left-symmetric color dialgebras and Hom-tridendriform
 color algebras. Next, we study their connection with Hom-associative color algebra, Hom-post-Lie color algebra and Hom-Poisson color dialgebras.
Finally, we generalize Yau's twisting to a class of color Hom-algebras and used endomorphisms
or elements of centroids to produce other color Hom-algebras from given one. 
\end{abstract}

% \date{}
% \noindent {\small Revised: April 16, 2014, June 28, 2014 $\parallel$  Accepted: July 4, 2014}

% ------------ [Running Heads - for odd and even pages] - please insert it only on page 2!
%\pagestyle{fancy} \fancyhead{} \fancyhead[EC]{Hom-left-symmetric color dialgebras, Hom-tridendriform color algebras and  Yau's twisting generalizations}
%\fancyhead[EL,OR]{\thepage} \fancyhead[OC]{Ibrahima bakayoko and Sergei Silvestrov} \fancyfoot{}
%\renewcommand\headrulewidth{0.5pt}
%------------------------------------------------------------------------------

\section{Introduction}

The investigations of various $q$-deformations (quantum deformations) of Lie algebras began a period of rapid expansion in 1980's stimulated by introduction of quantum groups motivated by applications to the quantum Yang-Baxter equation, quantum inverse scattering methods and constructions of the quantum deformations of universal enveloping algebras of semi-simple Lie algebras. In
\cite{AizawaSaito,ChaiElinPop,ChaiIsLukPopPresn,ChaiKuLuk,ChaiPopPres,CurtrZachos1,DamKu,DaskaloyannisGendefVir,Hu,Kassel92,LiuKQuantumCentExt,LiuKQCharQuantWittAlg,LiuKQPhDthesis}
various versions of $q$-deformed Lie algebras appeared in physical contexts such as string theory, vertex models in conformal field theory, quantum mechanics and quantum field theory in the context of $q$-deformations of infinite-dimensional algebras, primarily the $q$-deformed Heisenberg algebras, $q$-deformed oscillator algebras and $q$-deformed Witt and $q$-deformed Virasoro algebras, and some interesting $q$-deformations of the Jacobi identity for Lie algebras in these $q$-deformed algebras were observed.

Hom-Lie algebras and more general quasi-Hom-Lie algebras were introduced first by Larsson, Hartwig and Silvestrov \cite{HLS}, where
the general quasi-deformations and discretizations of Lie algebras of vector fields using more general $\sigma$-derivations (twisted derivations) and a general method for construction of deformations of Witt and Virasoro type algebras based on twisted derivations have been developed,
initially motivated by the $q$-deformed Jacobi identities observed for the $q$-deformed algebras in physics, along with $q$-deformed versions of homological algebra and discrete modifications of differential calculi. The general abstract quasi-Lie algebras and the subclasses of quasi-Hom-Lie algebras and Hom-Lie algebras as well as their color algebras generalizations have been introduced \cite{HLS,LS1,LS2,LSGradedquasiLiealg,LS3,Czech:witt}.
Subsequently, various classes of Hom-Lie admissible algebras have been considered in \cite{ms:homstructure}. In particular, in \cite{ms:homstructure}, the Hom-associative algebras have been introduced and shown to be Hom-Lie admissible, that is leading to Hom-Lie algebras using commutator map as new product, and in this sense constituting a natural generalization of associative algebras, as Lie admissible algebras leading to Lie algebras via commutator map as new product.
In \cite{ms:homstructure}, moreover several other interesting classes of Hom-Lie admissible algebras generalising some classes of non-associative algebras, as well as examples of finite-dimensional Hom-Lie algebras have been described. Since these pioneering works \cite{HLS,LS1,LS2,LSGradedquasiLiealg,LS3,ms:homstructure}, Hom-algebra structures have developed in a popular broad area with increasing number of publications in various directions.
In Hom-algebra structures, defining algebra identities are twisted by linear maps. Hom-algebras structures are very useful since Hom-algebra structures of a given type include their classical counterparts and open more possibilities for deformations, extensions of homology and cohomology structures and representations, Hom-coalgebra, Hom-bialgebras and Hom-Hopf algebras (see for example
\cite{AbAmMakh:CohomHomcoloralg,AmEjMakh:homdeformation,AmMakh:HomLieadmsuperalg,AmMakhSaad:cohomhomliesuperalgqdefwittalg,ams:ternary,ams:naryhomnambulie,AtMaSi:GenNambuAlg,
BI:LaplacehomLiequasibialg,BI:LmodcomodhomLiequasibialg,Hombiliform,HounkonnouDasundo:homcentersymalgbialgPropConseq2016,HounkonnouDasundo:homcentersymalgbialg2018,
ElhamdadiMakhlouf:DeformHomAltHomMalcev,kms:nhominduced,LarssonSigSilvJGLTA2008,LS1,MakSil:HomLieAdmissibleHomCoalgHomHopf,RichardSilvestrovJA2008,shenghomrep,SigSilvGLTbdSpringer2009,YauHomEnv,YauGenCom,YauHomNambuLie} and references therein).

Dendriform algebras were introduced by Loday as algebras with two operations which dichotomize the notion of associative algebras \cite{JLodayDialgebras,JLoday1DialgebrasPrepr}.
They are connected to  K-theory, Hopf algebras, homotopy
Gerstenhaber algebra, operads, homology, combinatorics and quantum field theory where they occur in
the theory of renormalization of Connes and Kreimer. While tridendriform algebra were
introduced later by Loday and Ronco in their study of polytopes and Koszul duality\cite{JL2}.

Hom-tridendriform algebras were introduced in \cite{MakhloufHomdemdoformRotaBaxterHomalg2011} as a twisted generalization of tridendriform algebras. More precisely,
a Hom-tridendriform  algebra is a linear space $A$, together with
three bilinear maps  $\dashv, \vdash, \cdot : A\otimes A\rightarrow A$ and  a linear map $\alpha : A\rightarrow A$ satisfing, for $x, y, z\in A$,
\begin{eqnarray}
(x\dashv y)\dashv\alpha(z)&=&\alpha(x)\dashv(y\dashv z+y\vdash z+y\cdot z),\quad
\alpha(x)\vdash(y\vdash z)=(x\dashv y+x\vdash y+x\cdot y)\vdash\alpha(z),\nonumber\\
(x\dashv y)\cdot\alpha(z)&=&\alpha(x)\cdot(y\vdash z),\quad (x\vdash y)\cdot\alpha(z)=\alpha(x)\vdash(y\cdot z),\quad
(x\cdot y)\dashv\alpha(z)=\alpha(x)\cdot(y\dashv z),\nonumber\\
 (x\cdot y)\cdot\alpha(z)&=&\alpha(x)\cdot(y\cdot z),\quad (x\vdash y)\dashv\alpha(z)=\alpha(x)\vdash(y\dashv z)\nonumber.
\end{eqnarray}
When $\alpha=id$, we recover tridendriform algebras. The author provides constructions of these algebras and their relationships with Hom-preLie
algebras.

Connections between Hom-bialgebras and Hom-coalgebras, Hom-dendiform and tridendiform systems and Rota-Baxter and Hom-Rota-Baxter Hom-algebra structures are considered in \cite{MaZheng:RotaBaxtMonoidalHomAlg,MaMakSilv:RotaBaxterCosystCoquasitriangmixedbialg,MaMakSilv:RotaBaxterbisystemscovbialg,MaMakSilv:CurvedOoperatorsyst,MakhloufHomdemdoformRotaBaxterHomalg2011,MakYau:RotaBaxterHomLieadmis}.

As for Hom-Post-Lie algebras, they has been studied in \cite{BI} as a twisted generalization post-Lie algebra.
Post-Lie algebras first arise form the work of Bruno Vallette \cite{BV:Homologygenpartitposets} in 2007 through the purely operadic technique of Koszul dualization.
In \cite{ML:postLiealg}, it shown that they also arise naturally from differential geometry of homogeneous spaces and Klein geometries, topics that are
closely related to Cartan's method of moving frames. The universal
enveloping algebras of post-Lie algebras and the free post-Lie algebra are studied. Some examples and related structures are given.
Further, they were extended to graded case in \cite{IT}.

It have been noticed by D. Yau that one can deform associative and Lie algebra structures into Hom-associative and Hom-Lie algebras by composition of associative algebra multiplication with algebra endomorphism. This procedure, known as Yau's twisting or a composition method has been intensively used for many other Hom-algebraic structures such as
Hom-Akivis algebras \cite{Issa:Hom-Akivis algebras}, Hom-Lie Yamagutti algebras \cite{DN}, Hom-Poisson algebras \cite{DYa}, Hom-Post-lie algebra \cite{BI},
Hom-pre-Lie (or left Hom-symmetric)\cite{ms:homstructure,MakhloufHomdemdoformRotaBaxterHomalg2011}, G-Hom-associative algebra \cite{ms:homstructure,YauHomHom}, Hom-dendriform algebra \cite{MakhloufHomdemdoformRotaBaxterHomalg2011} \cite{YauHomHom}, module over Hom-bialgebras \cite{DYu} and so on.
Here we prove that this technique works for a big class of Hom-algebraic structures and for other special linear maps as averaging operator
or element of centroid. In the litterature, the graded version, including super and color case, of many of these structures were studied.

The paper is organized as follows.
In Section \ref{sec:homleftsymcolordialg}, we show on the one hand that the tensor product of two Hom-associative color algebra is also a Hom-associative color algebra.
On the other hand, that the tensor product of a commutative Hom-associative color algebra and a Hom-left-symmetric color algebra is a
Hom-left-symmetric color algebra. We give a construction of a Hom-associative color dialgebra from an associative color algebra and an averaging
operator. Next, we introduce Hom-left-symmetric color dialgebras and study their connection with Hom-Poisson color dialgebras.
In Section \ref{sec:homtridendriformcoloralgebras}, we show that Hom-Post-Lie and Hom-associative color algebras may come from Hom-tridendriform color algebras. Next, we pointed out
that the opposite of any Hom-tridendriform color algebra is also one, and any Hom-tridendriform color algebra carries a structure of Hom-dendriform
algebra. Section \ref{sec:Yaustwistinggeneralization} is devoted to a Yau's twisting generalization to Hom-algebraic structures with a finite number of bilinear products, and use averaging
operators or centroids to produce other  color Hom-algebras from given one.

Throughout this paper, all graded vector spaces are assumed to be over a field $\mathbb{K}$ of characteristic different from 2.

\section{Hom-left-symmetric color dialgebras}
\label{sec:homleftsymcolordialg}

Giving an abelian group $G$, a map $\varepsilon :G\times G\rightarrow {\bf \mathbb{K}^*}$ is called a skew-symmetric bicharacter or a commutation factor on $G$ if the following
identities hold for all $a, b, c\in G$:
\begin{itemize}
\item[i)] $\varepsilon(a, b)\varepsilon(b, a)=1$;
\item[ii)] $\varepsilon(a, b+c)=\varepsilon(a, b)\varepsilon(a, c)$;
\item[iii)] $\varepsilon(a+b, c)=\varepsilon(a, c)\varepsilon(b, c).$
\end{itemize}
Remark that
$\varepsilon(a, 0)=\varepsilon(0, a)=1, \varepsilon(a,a)=\pm 1$
for all $a\in G,$ where $0$ is the identity of $G$.
If $x$ and $y$ are two homogeneous elements of degree $a$ and $b$ respectively and $\varepsilon$ is a skew-symmetric bicharacter,
then we shorten the notation by writing $\varepsilon(x, y)$ instead of $\varepsilon(a, b)$.
\begin{definition}[\cite{LY}]
A Hom-associative color algebra is a triple $(A, \mu, \alpha) $ consisting of a $G$-graded linear space $A$, an even bilinear map $\mu : A\times A \rightarrow A $ and an
even homomorphism $\alpha: A \rightarrow A$ satisfying
\begin{eqnarray}
 \mu(\alpha(x), \mu(y, z))&=&\mu(\mu(x, y), \alpha(z)), \label{aca}
\end{eqnarray}
for any $x, y, z \in \mathcal{H}(A) $ (homogeneouse elements of $A$).

If in addition $\mu=\varepsilon(\cdot, \cdot)\mu^{op}$ i.e. $\mu(x, y)=\varepsilon(x, y)\mu(y, x)$, for any $x, y\in\mathcal{H}(A)$, the color Hom-associative algebra
$(A, \mu, \varepsilon, \alpha)$ is said to be a commutative Hom-associative color algebra.
\end{definition}

The below theorem is proved by a direct calculation.
\begin{theorem}\label{bk8}
 Let $(A_1, \cdot_1, \varepsilon, \alpha_1)$ and $(A_2, \cdot_2, \varepsilon, \alpha_2)$ be two Hom-associative color algebras.
Then $A=A_1\otimes A_2$ is endowed with a color Hom-associative algebra structure for twising map
$\alpha:=\alpha_1\otimes\alpha_2 : A\rightarrow A$ and the product $\ast : A\otimes A\rightarrow A$ defined by
\begin{eqnarray}
 \alpha(a_1\otimes a_2)&:=&\alpha_1(a_1)\otimes\alpha_2(a_2),\nonumber\\
(a_1\otimes a_2)\ast(b_1\otimes b_2)&:=&\varepsilon(a_2, b_1)(a_1\cdot_1 b_1)\otimes (a_2\cdot_2 b_2).\nonumber
\end{eqnarray}
\end{theorem}
\begin{corollary}
 The tensor product of two associative color algebras is also an associative color algebra.
\end{corollary}

\begin{definition}
A Hom-color algebra or a color Hom-algebra $(S, \cdot, \varepsilon, \alpha)$ is an $G$-graded linear space $S$ equipped with even bilinear multiplication $\cdot$, even twising map $\alpha$
and commutation factor $\varepsilon$. A Hom-color algebra is called a Hom-left-symmetric color algebra if the following Hom-left-symmetric color identity
(or {\it $\varepsilon$-Hom-left-symmetric identity})
\begin{eqnarray}
(x\cdot y)\cdot\alpha(z)-\alpha(x)\cdot(y\cdot z)=\varepsilon(x, y)\Big((y\cdot x)\cdot\alpha(z)-\alpha(y)\cdot(x\cdot z)\Big)\label{clsa}
\end{eqnarray}
is satisfied for all $x, y, z\in \mathcal{H}(S)$.
 \end{definition}
\begin{example}
 Let $A=A_{(-1)}\oplus A_{(1)}=<e_2, e_3>\oplus <e_1>$ be a $G=\{-1, +1\}$-graded linear space. Then $A$ is a Hom-left-symmetric color algebra
together with the bicharacter
$$\varepsilon(i, j)=(-1)^{(i-1)(j-1)/4},$$
the multiplication
$$e_1e_1=-e_1,\quad e_2e_1=-ae_2,\quad e_3e_1=e_3, \quad e_1e_2=-ae_2, $$
 and the even linear map $\alpha : A\rightarrow A$ defined by
$$\alpha(e_1)=e_1,\quad \alpha(e_2)=ae_2,\quad \alpha(e_3)=-e_3, \quad a\in\mathbb{R}.$$
\end{example}

The proof of the below Theorem is proved by a straighforward computation.
\begin{theorem}
 Let $(S, \ast, \varepsilon, \alpha_S)$ be a Hom-left-symmetric color algebra and $(A, \cdot, \varepsilon, \alpha_A)$ a commutative
Hom-associative color algebra. \\
Then  $(S\otimes A, \circ, \varepsilon, \alpha_{S\otimes A})$ is a Hom-left-symmetric color algebra, with
\begin{eqnarray}
\alpha_{S\otimes A}&=&\alpha_S\otimes\alpha_A,\nonumber\\
 (x\otimes a)\circ(y\otimes b)&=&\varepsilon(a, y)(x\ast y)\otimes(a\cdot b),\nonumber
\end{eqnarray}
for all $x, y\in \mathcal{H}(S), a, b\in \mathcal{H}(A)$.
\end{theorem}
Recall that an ideal of a color Hom-algebra $(A, \cdot, \varepsilon, \alpha)$ is a graded subspace  $I$ such that $\alpha(I)\subset I$ and
$I\cdot A=A\cdot I\subset I$.
\begin{proposition}
Let $(S, \cdot, \varepsilon, \alpha)$ be a Hom-left-symmetric color algebra and $I$ an ideal of $S$ such that
 for any $i, j\in\mathcal{H}(I)$ and $a, b, c, d\in\mathcal{H}(S)$,
\begin{eqnarray}
\alpha(i)\cdot(a\cdot b)-(i\cdot a)\cdot\alpha(b)&=&\varepsilon(i, a)\Big(\alpha(a)\cdot(i\cdot b)-(a\cdot i)\cdot\alpha(b)\Big),\nonumber\\
 \alpha(c)\cdot(d\cdot j)-(c\cdot d)\cdot\alpha(j)&=&\varepsilon(c, d)\Big(\alpha(d)\cdot(c\cdot j)-(d\cdot c)\cdot\alpha(j)\Big).\nonumber
\end{eqnarray}
Then $(I\oplus S, \dashv, \vdash, \alpha_{I\oplus S})$ is a Hom-left-symmetric color dialgebra with
\begin{eqnarray}
 \alpha_{I\oplus S}&=&\alpha_I\oplus\alpha_S,\nonumber\\
(i_1+a_1)\dashv(i_2+a_2)&=&i_1a_2+a_1a_2, \nonumber\\
(i_1+a_1)\vdash(i_2+a_2)&=&a_1i_2+a_1a_2, \nonumber
\end{eqnarray}
for $i_1, a_1\in S_p, i_1, a_2\in S_q$.
\end{proposition}
\begin{proof}
It is straighforward by calculation.
\end{proof}
\begin{definition}\label{gls}
A Hom-left-symmetric color dialgebra is a $G$-graded linear space $S$ equipped with a bicharacter $\varepsilon : G\otimes G\rightarrow \mathbb{K}^*$
 on $G$ and two even bilinear products $\dashv : S\times S\rightarrow S$ and $\vdash : S\times S\rightarrow S$ satisfying the identities
\begin{eqnarray}
 \alpha(x)\dashv(y\dashv z)&=&\alpha(x)\dashv(y\vdash z),\label{als1}\\
(x\vdash y)\vdash \alpha(z)&=&(x\dashv y)\vdash\alpha(z),\label{als2}\\
\alpha(x)\dashv(y\dashv z)-(x\dashv y)\dashv \alpha(z)&=&\varepsilon(x, y)\Big(
\alpha(y)\vdash(x\dashv z)-(y\vdash x)\dashv \alpha(z)\Big),\label{als3}\\
\alpha(x)\vdash(y\vdash z)-(x\vdash y)\vdash \alpha(z)&=&\varepsilon(x, y)\Big(
\alpha(y)\vdash(x\vdash z)-(y\vdash x)\vdash \alpha(z)\Big),\label{als4}
\end{eqnarray}
for all $x, y, z\in \mathcal{H}(S)$.
\end{definition}
\begin{remark}
 Relation \eqref{als4} means that $(S, \vdash, \varepsilon, \alpha)$ is a Hom-left-symmetric color algebra. So any Hom-left-symmetric color
algebra is a Hom-left-symmetric color dialgebra.
\end{remark}
 \begin{definition}[\cite{BD}] \label{dia}
 A Hom-associative color dialgebra is a quintuple $(D, \dashv, \vdash, \varepsilon, \alpha)$, where $D$ is a $G$-graded linear space,
$\dashv, \vdash : D\otimes D\rightarrow D$ are even bilinear maps, $\varepsilon : G\otimes G\rightarrow \mathbb{K}^*$ is a bicharacter
 and $\alpha : D\rightarrow D$ is an even linear map such that the following  five axioms
 \begin{eqnarray}
(x\vdash y)\dashv\alpha(z)&\stackrel{}{=}&\alpha(x)\vdash(y\dashv z), \label{dia1}\\
  \alpha(x)\dashv (y\dashv z)&\stackrel{}{=}&(x\dashv y)\dashv\alpha(z),\label{dia2}\\
(x\dashv y)\dashv\alpha(z)&=&\alpha(x)\dashv(y\vdash z),\label{dia3}\\
(x\vdash y)\vdash\alpha(z)&=&\alpha(x)\vdash(y\vdash z),\label{dia4}\\
\alpha(x)\vdash(y\vdash z)&=&(x\dashv y)\vdash\alpha(z),\label{dia5}
 \end{eqnarray}
are satisfied for $x, y, z\in \mathcal{H}(D)$.
\end{definition}
\begin{remark}\label{bk4}
 If $(A, \dashv, \vdash, \varepsilon, \alpha)$ is a Hom-associative color dialgebra in which $\dashv=\vdash=:\mu$, then
$(A, \mu, \varepsilon, \alpha)$ is a Hom-associative  color algebra. Conversely, any Hom-associative color algebra $(A, \mu, \varepsilon, \alpha)$
is a Hom-associative color dialgebra with $\dashv:=\mu=:\vdash$.
\end{remark}

\begin{lemma}\label{iibb}
 A Hom-left-symmetric color dialgebra $S$ is a Hom-associative color dialgebra if and only if both products of $S$ are color Hom-associative.
\end{lemma}
\begin{proof}
If a Hom-left-symmetric color dialgebra $S$ is a Hom-associative color dialgebra, then both products $\dashv$ and $\vdash$ defined over $S$
are Hom-associative according to Definition \ref{dia}.
Conversely, if each product of a Hom-left-symmetric color dialgebra
 is Hom-associative, then from \eqref{als3}, we get axiom \eqref{als1}.
\end{proof}

% Now we introduce averaging operator in order to produce Hom-associative color dialgebras from Hom-associative color algebras.
We need the below definition for he next theorem.
\begin{definition}
 A Nijenhuis operator over a Hom-associative color algebra  $(A, \mu, \varepsilon, \alpha)$ is an even linear map $N : A\rightarrow A$ such
that $\alpha\circ N=N\circ\alpha$ and
\begin{eqnarray}
 \mu(N(x), N(y))=N\Big(\mu(N(x), y)+\mu(x, N(y))-N(\mu(x, y))\Big),
\end{eqnarray}
for all $x, y\in\mathcal{H}(A)$.
 \end{definition}
\begin{theorem}
 Let $(A, \mu, \varepsilon, \alpha)$ be a Hom-associative color algebra and $N: A\rightarrow A$ a Nijenhuis operator. Then the new multiplication
$\mu_N : A \rightarrow A$ given by
$$\mu^N(x, y)=\mu(N(x), y)+\mu(x, N(y))-N(\mu(x, y)),$$
makes $A$ into a Hom-associative color algebra.
\end{theorem}
% \begin{proof}
%  d
% \end{proof}
% \begin{example}
%  f
% \end{example}
\begin{corollary}
 Let $(A, \mu, \varepsilon, \alpha)$ be an associative color algebra, $\alpha : A\rightarrow A$ be an even endomorphism and $N: A\rightarrow A$
be a Nijenhuis operator commuting with $\alpha$. Then
$\mu_\alpha^N=\alpha\circ \mu^N$
is a Hom-associative color algebra.
\end{corollary}

Now we introduce averaging operator in order to produce Hom-associative color dialgebras from Hom-associative color algebras.
\begin{definition}
 An averaging operator over a Hom-associative color algebra  $(A, \mu, \varepsilon, \alpha)$ is an even linear map $\beta : A\rightarrow A$ such
that $\alpha\circ\beta=\beta\circ\alpha$ and
\begin{eqnarray}
 \beta(\mu(\beta(x), y)=\mu(\beta(x), \beta(y))=\beta(\mu(x, \beta(y))), \label{avo1}
\end{eqnarray}
for all $x, y\in\mathcal{H}(A)$.
 \end{definition}

\begin{theorem}\label{bk2}
 Let $(A, \cdot, \varepsilon)$ be an associative color algebra and $\alpha : A\rightarrow A$ an averaging operator such that
$(A, \cdot, \varepsilon, \alpha)$ be a Hom-associative color algebra. For any $x, y\in\mathcal{H}(A)$, define new operations on $A$ by
$$x\vdash y=\alpha(x)\cdot y\quad\mbox{and}\quad x\dashv y:=x\cdot\alpha(y).$$
Then $(A, \dashv, \vdash, \varepsilon, \alpha)$ is a Hom-associative color dialgebra.
\end{theorem}
\begin{proof}
 We prove only one axiom. The others being proved similarly. For any $x, y, z\in \mathcal{H}(A)$,
\begin{eqnarray}
 \alpha(x)\dashv(y\dashv z)-(x\dashv y)\dashv\alpha(z)
&=&\alpha(x)\cdot\alpha(y\cdot\alpha(z))-(x\cdot\alpha(y))\cdot\alpha^2(z)\nonumber\\
&=& \alpha(x)\cdot(\alpha(y)\cdot\alpha(z))-(x\cdot\alpha(y))\cdot\alpha^2(z)\quad\;(\mbox{by}\;\eqref{avo1})\nonumber\\
&=& \alpha(x)\cdot(\alpha(y)\cdot\alpha(z))-\alpha(x)\cdot(\alpha(y)\cdot\alpha(z))\;\;(\mbox{by}\;\eqref{aca})\nonumber\\
&=&0.\nonumber
\end{eqnarray}
This proves the axiom.
\end{proof}

\begin{definition}[\cite{BD}]
A Hom-Poisson color  dialgebra is a sextuple $(P, \dashv, \vdash, \{-, -\}, \varepsilon, \alpha)$ in which $P$ is a $G$-graded linear space,
 $\dashv, \vdash, \{-, -\} : P\otimes P\rightarrow P$ are three even bilinear maps, $\varepsilon : G\otimes G\rightarrow \mathbb{K}^*$ is
a bicharacter and  $\alpha : P\rightarrow P$ is an even linear map such that
 \begin{eqnarray}
  \{x\dashv y, \alpha(z)\}&=&\alpha(x)\dashv\{y, z\}+\varepsilon(y, z)\{x, z\}\dashv\alpha(y),\label{hpd1}\\
\{x\vdash y, \alpha(z)\}&=&\alpha(x)\vdash\{y, z\}+\varepsilon(y, z)\{x, z\}\vdash\alpha(y),\label{hpd2}\\
\{\alpha(x), y\dashv z\}&=&\varepsilon(x, y)\alpha(y)\vdash\{x, z\}+\{x, y\}\dashv\alpha(z)=\{\alpha(x), y\vdash z\},\label{hpd3}
 \end{eqnarray}
for all $x, y, z\in\mathcal{H}(P)$.
 \end{definition}
\begin{lemma}[\cite{BD}] \label{diaxx}
 Let $(D, \dashv, \vdash, \varepsilon, \alpha)$ be a Hom-associative color dialgebra. \\
 Then $(D, \dashv, \vdash, \{-, -\}, \varepsilon, \alpha)$ is a
  Hom-Poisson color dialgebra, where
\begin{eqnarray}
 \{x, y\}=x\dashv y-\varepsilon(x, y) y\vdash x,\nonumber
\end{eqnarray}
for any $x, y\in \mathcal{H}(D)$.
\end{lemma}
\begin{theorem}
 Let $(D, \dashv, \vdash, \varepsilon, \alpha)$ be a Hom-left-symmetric color dialgebra with Hom-associative products.
 Then $(D, \dashv, \vdash, \{-, -\}, \varepsilon, \alpha)$ is a  Hom-Poisson color dialgebra, where
\begin{eqnarray}
 \{x, y\}=x\dashv y-\varepsilon(x, y) y\vdash x,\nonumber
\end{eqnarray}
for any $x, y\in \mathcal{H}(D)$.
\end{theorem}
\begin{proof}
 It follows from Lemma \ref{iibb} and Lemma \ref{diaxx}.
\end{proof}
\section{Hom-tridendriform color algebras}
\label{sec:homtridendriformcoloralgebras}

We introduce  Hom-tridendriform color algebras, give some proprities and study their connection with Hom-post-Lie color algebras and color
 Hom-associative algebras.
\begin{definition}\label{tdd}
 A  Hom-tridendriform color algebra is a sextuple $(T, \dashv, \vdash, \cdot, \varepsilon, \alpha)$ consisting of a $G$-graded linear space $T$,
 three even bilinear maps  $\dashv, \vdash, \cdot : T\otimes T\rightarrow T$, a bicharacter $\varepsilon : G\times G\rightarrow\mathbb{K}^*$
 and  an even linear map $\alpha : T\rightarrow T$ satisfing
\begin{eqnarray}
(x\dashv y)\dashv\alpha(z)&=&\alpha(x)\dashv(y\dashv z+\varepsilon(z, y)y\vdash z+\varepsilon(z, y)y\cdot z),\label{t1}\\
(x\vdash y)\dashv\alpha(z)&=&\varepsilon(z, x)\alpha(x)\vdash(y\dashv z),\\
\alpha(x)\vdash(y\vdash z)&=&(\varepsilon(x, y)x\dashv y+x\vdash y+x\cdot y)\vdash\alpha(z),\\
(x\dashv y)\cdot\alpha(z)&=&\varepsilon(y, x)\alpha(x)\cdot(y\vdash z),\\
(x\vdash y)\cdot\alpha(z)&=&\alpha(x)\vdash(y\cdot z),\\
(x\cdot y)\dashv\alpha(z)&=&\varepsilon(z, x)\alpha(x)\cdot(y\dashv z),\\
(x\cdot y)\cdot\alpha(z)&=&\alpha(x)\cdot(y\cdot z),
\end{eqnarray}
for $x, y, z\in \mathcal{H}(T)$.
 \end{definition}
\begin{remark}
When the color Hom-associative  product is identically null, we get a Hom-dendriform color algebra \cite{BD}.
\end{remark}
\begin{example}
 Let $A=A_0\oplus A_1=<e_1, e_2>\oplus<e_3>$ be a three-dimensional superspace. The multiplications
\begin{eqnarray}
 e_2\dashv e_2&=& e_2\vdash e_2=ae_1, \quad e_2\cdot e_2=-ae_1, \nonumber\\
e_3\dashv e_3&=& e_3\vdash e_3=be_1, \quad e_3\cdot e_3=be_1, \nonumber
\end{eqnarray}
and the even linear map $\alpha : A\rightarrow A$ defined by
$$\alpha(e_1)=e_1, \quad \alpha(e_2)=e_1+e_2, \quad \alpha(e_3)=-e_3,$$
make $A$ into a Hom-tridendriform color algebra, for any $a, b\in\mathbb{R}$.
\end{example}

\begin{proposition}
 Let $(T, \dashv, \vdash, \cdot, \varepsilon, \alpha)$ be a  Hom-tridendriform color algebra. \\ Then
 $(T, \dashv^{op}, \vdash^{op}, \cdot^{op}, \varepsilon, \alpha)$  is also a  Hom-tridendriform color algebra, with
$$x\dashv^{op}y:=y\vdash x,\quad x\vdash^{op}y:=y\dashv x,\quad x\cdot^{op}y:=y\cdot x,$$
for any $x, y\in \mathcal{H}(T)$.
\end{proposition}
\begin{proof}
 The proof is straighforward by calculation by using axioms in Definition \ref{tdd}.
\end{proof}

\begin{proposition}
Let $(A, \dashv, \vdash, \cdot, \varepsilon, \alpha)$ be a  Hom-tridendriform color algebra. \\
Then $(A, \dashv, \vdash', \varepsilon, \alpha)$ is a Hom-dendriform color algebra, where
$$x\vdash' y:=x\vdash y+x\cdot y,$$
for any $x, y\in\mathcal{H}(T)$.
\end{proposition}
\begin{proof}
 It comes immediately from Definition  \ref{tdd}.
\end{proof}
%%%%%%%%%%%%%%%%%%%%%%%%%%%%%%%%%%%%%%
% \begin{definition}
% An even linear map $R : A\rightarrow A$ on a Hom-tridendriform color algebra $(A, \dashv, \vdash, \cdot, \varepsilon)$ is said to be a
% Rota-Baxter operator of weight $\lambda\in\mathbb{K}$ if $R\circ \alpha=\alpha\circ R$ and
% \begin{eqnarray}
%  R(x)\dashv R(y)&=&R(R(x)\dashv y+ x\dashv R(y)+\lambda x\dashv y),\\
%  R(x)\vdash R(y)&=&R(R(x)\vdash y+x\vdash R(y)+\lambda x\vdash y),\\
%  R(x)\cdot R(y)&=&R(R(x)\cdot y+x\cdot R(y)+\lambda x\cdot y),\label{initl}
% \end{eqnarray}
% for any $x, y\in\mathcal{H}(D)$. \\
% In this case, the quintuple $(A, \dashv, \vdash, \cdot, R)$ is called {\it Rota-Baxter Hom-tridendriform color  algebra of weight $\lambda$}.
% \end{definition}

We need the following definition for the next proposition.
\begin{definition}
 Let $(A, \cdot, \varepsilon, \alpha)$ be a Hom-associative color algebra and $\lambda \in \mathbb{K}$. An even linear map $R : A\rightarrow A$
 is called a Rota-Baxter operator of weight $\lambda$ on $A$ if it satisfies the identities
\begin{eqnarray}
R\circ \alpha&=&\alpha\circ R,\\
R(x)\cdot R(y) &=& R\Big(R(x)\cdot y + x\cdot R(y) +\lambda x\cdot y\Big),
% \quad (\mbox{\it Rota-Baxter identity}) \label{rbi}
\end{eqnarray}
for any $x, y\in\mathcal{H}(A)$.
\end{definition}

\begin{theorem}\label{car1}
 Let $(A, \cdot, \varepsilon, \alpha, R)$ be a Rota-Baxter Hom-associative color algebra of weight $\lambda$.
Let us define three new operations $\dashv, \vdash$ and
$\ast$ on $A$ by
$$x\dashv y:=x\cdot R(y), \quad x\vdash y:=\varepsilon(x, y)R(x)\cdot y\quad\mbox{and}\quad x\ast y:=\lambda\varepsilon(x, y)x\cdot y.$$
Then $(A, \dashv, \vdash, \ast, \varepsilon, \alpha)$ is a Hom-tridendriform color algebra.
\end{theorem}
% \begin{proof}
%  The proof follows from a simple calculation. For instance, for axiom \eqref{initfh}, one has :
% \begin{eqnarray}
%  (x\dashv y)\vdash-\varepsilon(z, x)x\vdash (y\dashv z)
% &=&\varepsilon(x, y)(R(x)\cdot y)\cdot R(z)-\varepsilon(z, x)\varepsilon(x, y+z) R(x)\cdot(y\cdot R(z))\nonumber\\
% &=&\varepsilon(x, y)\Big((R(x)\cdot y)\cdot R(z)- R(x)\cdot(y\cdot R(z)) \Big).\nonumber
% \end{eqnarray}
% The last line vanishes by associativity.
% \end{proof}
\begin{example}
Let $G=\{-1, +1\}$ be an abelian group and $A=A_{(-1)}\oplus A_{(1)}=<e_2>\oplus<e_1>$ a $G$-graded two-dimensional linear space. The quintuple
$(A, \cdot, \varepsilon, \alpha, R)$ is a Rota-Baxter Hom-associative color algebra of weight $\lambda$ with
\begin{itemize}
 \item the multiplication, $e_1\cdot e_1=-e_1,\quad e_1\cdot e_2=e_2,\quad e_2\cdot e_1=e_2,\quad e_2\cdot e_2=e_1$,
\item the bicharacter, $\varepsilon(i, j)=(-1)^{(i-1)(j-1)/4}$,
\item the even linear map $\alpha : A\rightarrow A$ defined by : $\alpha(e_1)=e_1,\quad \alpha(e_2)=-e_2$,
\item the Rota-Baxter operator $R : A\rightarrow A$ given by : $R(e_1)=-\lambda e_1, R(e_2)=-\lambda e_2$.
\end{itemize}
Therefore, $(A, \dashv, \vdash, \ast, \varepsilon, \alpha)$ is a Hom-tridendriform color algebra with
\begin{center}
\begin{tabular}{lll}
 $e_1\dashv e_1=\lambda e_1$,& $e_1\vdash e_1=\lambda e_1$& $e_1\ast e_1=-\lambda e_1$,     \\
$e_1\dashv e_2=-\lambda e_2$,&    $e_1\vdash e_2=-\lambda e_2$,& $e_1\ast e_2=\lambda e_2$,\\
$e_2\dashv e_1=-\lambda e_2,$&     $ e_2\vdash e_1=-\lambda e_2$& $e_2\ast e_1=\lambda e_2$,\\
$e_2\dashv e_2=-\lambda e_1,$&      $ e_2\vdash e_2=\lambda e_1,$& $e_2\ast e_2=-\lambda e_1$.
\end{tabular}
\end{center}

% \begin{eqnarray}
%  e_1\dashv e_1&=&\lambda e_1, \quad   e_1\dashv e_1=\lambda          \nonumber\\
% e_1\dashv e_2&=&-\lambda e_2, \quad    e_1\dashv e_2=-\lambda e_2,\nonumber\\
% e_2\dashv e_1&=&-\lambda e_2, \quad     e_2\dashv e_1=-\lambda \nonumber\\
% e_2\dashv e_2&=&-\lambda e_1, \quad       e_2\dashv e_2=-\lambda e_1,\nonumber
% \end{eqnarray}

\end{example}

\begin{proposition}
 Let $(A, \cdot, \varepsilon, R, \alpha)$ be a Rota-Baxter Hom-associative color algebra of weight $\lambda$.
Define the operations $\dashv$ and $\vdash$ by
$$x\dashv y:=x\cdot R(y)+\lambda x\cdot y \quad\mbox{and}\quad x\vdash y:=\varepsilon(x, y)R(x)\cdot y.$$
Then $(A, \dashv, \vdash, \varepsilon)$ is a Hom-dendriform color algebra.
\end{proposition}
% \begin{proof}
% The axioms are checked as in the proof of Proposition \ref{car1}.
% \end{proof}
\begin{corollary}
 Let $(A, \cdot, \varepsilon, R)$ be a Rota-Baxter Hom-associative color algebra of weight $0$. We define the even bilinear operations
$\dashv : A\times A\rightarrow A$ and $\vdash : A\times A\rightarrow A$ on $A$ by
$$x\dashv y:=x\cdot R(y)\quad\mbox{and}\quad x\vdash y:=\varepsilon(x, y)R(x)\cdot y.$$
Then $(A, \dashv, \vdash, \varepsilon)$ is a Hom-dendriform color algebra.
\end{corollary}
% The below lemma asserts that Hom-tridendriform color algebras may give rise to Hom-associative color algebras.
% \begin{lemma}\cite{B2}\label{car22}
% Let $(A, \dashv, \vdash, \cdot, \varepsilon, \alpha)$ be a  Hom-tridendriform color algebra. Then $(A, \ast, \varepsilon, \alpha)$ is
% an associative color algebra, where
%  $x\ast y=x\vdash y+\varepsilon(x, y)x\dashv y+ x\cdot y$.
% \end{lemma}

%%%%%%%%%%%%%%%%%%%%%%%%%%%%%%%%%%%%%%%%%%%%%%%%%%%%%%%%%%%%%%%%%%%%%%%%%%%%%%%%%%%%%%%%%

In the below theorem, we associate a Hom-associative color algebra to any Hom-tridendriform color algebra.
\begin{theorem}\label{tp}
Let $(T, \dashv, \vdash, \cdot, \varepsilon, \alpha)$ be a  Hom-tridendriform color algebra. Then $(T, \ast, \varepsilon, \alpha)$ is
a Hom-associative color algebra, where
 $x\ast y=x\vdash y+\varepsilon(x, y)x\dashv y+ x\cdot y$.
\end{theorem}
\begin{proof}
For any $x, y, z\in\mathcal{H}(T)$
 \begin{eqnarray}
as_\alpha(x, y, z)&=&(x\vdash y)\vdash\alpha(z)+\varepsilon(x, y)(x\dashv y)\vdash\alpha(z)+(x\cdot y)\vdash\alpha(z)
+\varepsilon(x, z)\varepsilon(y, z)(x\vdash y)\dashv\alpha(z)\nonumber\\
&+&\varepsilon(x, z)\varepsilon(y, z)\varepsilon(x, y)(x\dashv y)\dashv\alpha(z)
+\varepsilon(x, z)\varepsilon(y, z)(x\cdot y)\dashv\alpha(z)+(x\vdash y)\cdot\alpha(z)\nonumber\\
&+&\varepsilon(x, y)(x\dashv y)\cdot\alpha(z)
+(x\cdot y)\cdot\alpha(z)-\alpha(x)\vdash(y\vdash z)-\varepsilon(y, z)\alpha(x)\vdash(y\dashv z)\nonumber\\
&-&\alpha(x)\vdash(y\cdot z)
-\varepsilon(x, y)\varepsilon(x, z)\alpha(x)\dashv(y\vdash z)
-\varepsilon(x, y)\varepsilon(x, z)\varepsilon(y, z)\alpha(x)\dashv(y\dashv z)\nonumber\\
&-&\varepsilon(x, y)\varepsilon(x, z)\alpha(x)\dashv(y\cdot z)-\alpha(x)\cdot(y\vdash z)-\varepsilon(y, z)\alpha(x)\cdot(y\dashv z)
-\alpha(x)\cdot(y\cdot z).\nonumber
 \end{eqnarray}
The left hand side vanishes by axioms in Definition \ref{tdd}.
This prove that $(A, \ast, \varepsilon, \alpha)$ is a Hom-associative color algebra.
This completes the proof.
\end{proof}
\begin{remark}
 Whenever $(T, \dashv, \vdash, \cdot, \varepsilon, \alpha)$ is commutative, we recover Lemma 3.3 \cite{IT}.
\end{remark}

\begin{definition}[\cite{BI1}]
  A  Hom-Poisson  color algebra consists of a $G$-graded linear space $A$, a multiplication $\mu : A\times A\rightarrow A$, an even bilinear bracket
 $\{\cdot, \cdot\} : A\times A\rightarrow A$ and
 an even linear map $\alpha : A\rightarrow A$ such that
\begin{enumerate}
 \item[(i)] $(A, \mu, \varepsilon, \alpha)$ is a  Hom-associative color algebra,
\item [(ii)]$(A, \{\cdot, \cdot\}, \varepsilon, \alpha)$ is a  Hom-Lie color algebra,
\item[(iii)]the color Hom-Leibniz identity is satisfied i.e.
\begin{eqnarray}
 \{\alpha(x), \mu(y, z)\}=\mu(\{x, y\}, \alpha(z))+\varepsilon(x,y)\mu(\alpha(y), \{x, z\}),\label{cca}
\end{eqnarray}
for any $x, y, z\in \mathcal{H}(A)$.
\end{enumerate}
% If in addition  $\mu$ is $\varepsilon$-commutative, the color Hom-Poisson algebra
% $(A, \mu, \{\cdot, \cdot\}, \varepsilon, \alpha)$ is said to be a $\varepsilon$-commutative color Hom-Poisson algebra.
\end{definition}

\begin{lemma}[\cite{BI1}]\label{fond}
Let $(A, \mu, \varepsilon, \alpha)$ be a  Hom-associative color algebra. \\
Then $(A, \mu, \{\cdot, \cdot\}=\mu-\varepsilon(\cdot, \cdot)\mu^{op}, \varepsilon, \alpha)$ is a  Hom-Poisson color algebra.
\end{lemma}

\begin{theorem}\label{car2}
 Let $(A, \dashv, \vdash, \cdot, \varepsilon, \alpha)$ be a Hom-tridendriform color algebra. \\
 Then $(A, \ast, [-, -], \varepsilon, \alpha)$ is a
 Hom-Poisson color algebra, where
$$x\ast y:=x\vdash y+\varepsilon(x, y)x\dashv y+x\cdot y\quad\mbox{and}\quad [x, y]:=x\ast y-\varepsilon(x, y)y\ast x.$$
\end{theorem}
\begin{proof}
 The proof follows from Theorem \ref{tp} and Lemma \ref{fond}.
\end{proof}
%
% \begin{corollary}
%   Let $(A, \cdot, \varepsilon, R)$ be a Rota-Baxter associative color algebra of weight $\lambda$. Then $A$ is a Hom-Poisson color
% algebra with respect to the products :
% $$x\ast y:=\varepsilon(x, y)\Big( R(x)\cdot y+x\cdot R(y)+\lambda x\cdot y\Big)\quad\mbox{and}\quad [x, y]:=x\ast y-\varepsilon(x, y)y\ast x.$$
% \end{corollary}
% \begin{proof}
%  It comes from propositions \ref{car1} and \ref{car2}.
% \end{proof}
\begin{definition}[\cite{IT}]\label{hplad}
 A Hom-post-Lie color algebra $(L, [-, -], \cdot, \varepsilon, \alpha)$ is a Hom-Lie color  algebra  $(L, [-, -], \varepsilon, \alpha)$, i.e.
\begin{eqnarray}
&& [x, y]=-\varepsilon(x, y)[y, x]\qquad\qquad\qquad (\varepsilon\mbox{-skew-symmetry})\label{ss}\\
&& \varepsilon(z, x)[\alpha(x), [y, z]] + \varepsilon(x, y)[\alpha(y), [z, x]]+\varepsilon(y, z)[\alpha(z), [x, y]]=0 \\
&& \qquad\qquad\qquad\qquad\qquad\qquad\qquad (\varepsilon\mbox{-Hom-Jacobi identity}) \label{chli}  \nonumber
\end{eqnarray}
together with an even bilinear map $\cdot : L\otimes L\rightarrow L$ such that
\begin{eqnarray}
&& \alpha(z)\cdot [x, y]-[z\cdot x, \alpha(y)]-\varepsilon(z, x)[\alpha(x), z\cdot y]=0,\qquad\qquad\qquad\qquad \label{pl4}\\
&& \alpha(z)\cdot(y\cdot x)-\varepsilon(z, y)\alpha(y)\cdot(z\cdot x)+\varepsilon(z, y)(y\cdot z)\cdot\alpha(x)
-(z\cdot y)\cdot\alpha(x) \\
&& \qquad\qquad\qquad\qquad\qquad\qquad\qquad\qquad +\varepsilon(z, y)[y, z]\cdot\alpha(x)=0,\label{pl3}  \nonumber
\end{eqnarray}
for any $x, y, z\in\mathcal{H}(L)$.
\end{definition}

\begin{example}
 Let $G=\mathbb{Z}_2\times\mathbb{Z}_2$ be an abelian group and $L$ be a tridimensional $G$-graded linear space defined by
$$L_{(0, 0)}=0, \quad L_{(0, 1)}=<e_2>, \quad L_{(1, 0)}=<e_1>, \quad L_{(1, 1)}=<e_3>.$$
Then $(L, [-, -], \cdot, \varepsilon, \alpha)$ is a Hom-Post-Lie color algebra with
\begin{itemize}
 \item the bicharacter : $\varepsilon((i_1, i_2), (j_1, j_2))=(-1)^{i_1j_1+i_2j_2}$,
\item the bracket : $[e_1, e_2]=-e_3, \quad [e_1, e_3]=-e_2, \quad [e_2, e_3]=-e_1$,
\item the multiplication :  $e_2e_1=-e_3, \quad e_2e_3=e_1$,
\item the even linear map $\alpha : L\rightarrow L$ given by : $\alpha(e_1)=-e_1, \quad \alpha(e_2)=-e_2, \quad \alpha(e_3)=e_3$.
\end{itemize}
\end{example}
% \begin{proposition}
%  Let $(A, \cdot, \alpha)$ be a left commutative Hom-Novikov color algebra i.e.
% \begin{eqnarray}
% (x\cdot y)\cdot\alpha(z)&=&\varepsilon(x, y)(y\cdot x)\cdot\alpha(z) \;\;  \mbox{(left commutativity)}\nonumber\\
%  (x\cdot y)\cdot\alpha(z)&=&\varepsilon(y, z)(x\cdot z)\cdot\alpha(y),\nonumber\\
% (x\cdot y)\cdot\alpha(z)-\alpha(x)(y\cdot z)&=&\varepsilon(x, y)\Big((y\cdot x)\cdot\alpha(z)-\alpha(y)\cdot(x\cdot z)\Big), \nonumber
% \end{eqnarray}
% for any $x, y, z\in\mathcal{H}(L)$. Then $(A, [-, -], \cdot, \varepsilon, \alpha)$ is a Hom-post-Lie color algeb
%  with $[x, y]=x\cdot y-\varepsilon(x, y)y\cdot x$.
% \end{proposition}
% \begin{proof}
%  to be proved
% \end{proof}

 \begin{proposition}
 Let $(L, [-, -], \cdot, \varepsilon, \alpha)$ be a color post-Hom-Lie algebra. With the product
$$x\ast y=x\cdot y+\frac{1}{2}[x, y],$$
 $(L, \ast, \varepsilon, \alpha)$ is a color Hom-Lie admissible algebra, that is the map $\ast : L\times L\rightarrow L$ satisfies the
$\varepsilon$-Hom-Jacobi identity.
\end{proposition}
\begin{proof}
 It follows from direct computation by using axioms in Definition \ref{hplad}.
\end{proof}

To any Hom-tridendriform color algebra one can associate a Hom-post-Lie color algebra, as stated in the following result.
\begin{theorem}
 Let $(T, \dashv, \vdash, \cdot, \varepsilon, \alpha)$ be a  Hom-tridendriform color algebra. \\ Then $(T, \circ, [-, -], \varepsilon, \alpha)$
 is a Hom-post-Lie color algebra, where
$x\circ y=x\vdash y-y\dashv x$ and $[x, y]=x\cdot y-\varepsilon(x, y)y\cdot x$, for any $x, y\in\mathcal{H}(T)$.
\end{theorem}
\begin{proof}
The $\varepsilon$-skew-symmetry and the $\varepsilon$-Hom-Jacobi identity are trivial, by Lemma \ref{fond}. Now, for any $x, y, z\in\mathcal{H}(T)$,
 we have
 \begin{eqnarray}
&& \hspace{-0.5cm} \alpha(x)\circ(y\circ z)-\varepsilon(x, y)\alpha(y)\circ(x\circ z)+\varepsilon(x, y)(y\circ x)\circ\alpha(z)-(x\circ y)\circ \alpha(z)
+\varepsilon(x, y)[y, x]\circ\alpha(z)=\nonumber\\
&&=\alpha(x)\vdash(y\vdash z)-\alpha(x)\vdash(z\dashv y)-(y\vdash z)\dashv\alpha(x)+(z\dashv y)\dashv\alpha(x)\nonumber\\
&&\quad-\varepsilon(x, y)\alpha(y)\vdash(x\vdash z)+\varepsilon(x, y)\alpha(y)\vdash(z\dashv x)
+\varepsilon(x, y)(x\vdash z)\dashv\alpha(y)-\varepsilon(x, y)(z\dashv x)\dashv\alpha(y)\nonumber\\
&&\quad+\varepsilon(x, y)(y\vdash x)\vdash\alpha(z)-\varepsilon(x, y)(x\dashv y)\vdash\alpha(z)-\varepsilon(x, y)\alpha(z)\dashv(y\vdash x)
+\varepsilon(x, y)\alpha(z)\dashv(x\dashv y)\nonumber\\
&&\quad-(x\vdash y)\vdash\alpha(z)+(y\dashv x)\vdash\alpha(z)+\alpha(z)\dashv(x\vdash y)-\alpha(z)\dashv(y\dashv x)\nonumber\\
&&\quad+\varepsilon(x, y)(y\cdot x)\vdash\alpha(z)-(x\cdot y)\vdash\alpha(z)-\varepsilon(x, y)\alpha(z)\dashv(y\cdot x)
+\alpha(z)\dashv(x\cdot y).\nonumber
 \end{eqnarray}
Which vanishes by axioms in Definition \ref{tdd}.
 Next,
\begin{eqnarray}
 &&\alpha(z)\circ[x, y]-[z\circ x, \alpha(y)]-\varepsilon(z, x)[\alpha(x), z\circ y]=\nonumber\\
&&=\alpha(z)\vdash(x\cdot y)-\varepsilon(x, y)\alpha(z)\vdash(y\cdot x)-(x\cdot y)\dashv\alpha(z)
+\varepsilon(x, y)(y\cdot x)\dashv \alpha(z)\nonumber\\
&&\quad-(z\vdash x)\cdot\alpha(y)+(x\dashv z)\cdot\alpha(y)+\varepsilon(x+z, y)\alpha(y)\cdot(z\vdash x)
-\varepsilon(x+z, y)\alpha(y)\cdot(x\dashv z)\nonumber\\
&&\quad-\varepsilon(z, x)\alpha(x)\cdot(z\vdash y)+\varepsilon(z, x)\alpha(x)\cdot(y\dashv z)+\varepsilon(x, y)(z\vdash y)\cdot\alpha(x)
-\varepsilon(x, y)(y\dashv z)\cdot\alpha(x)\nonumber.
\end{eqnarray}
The left hand side vanishes by axioms in Definition \ref{tdd}.
\end{proof}

\section{Generalization of Yau's twisting composition methos}
\label{sec:Yaustwistinggeneralization}
In this section, we generalize Yau's twisting to a large class of color Hom-algebras and use the centroids to produce other color Hom-algebras from
given one. To state the fundamental result, Theorem \ref{gftp1}, of this section, we give the following definitions.
\begin{definition}
1) By a color Hom-algebra we mean a $(n+3)$-uple $(A, \mu_1, \dots, \mu_n, \varepsilon, \alpha)$ in which $A$ is a $G$-graded linear space,
$\mu_i : A\otimes A \rightarrow A$ $(i=1, \dots, n)$ are even bilinear maps, $\varepsilon : G\times G\rightarrow\mathbb{K}^*$ is a bicharacter
 and $\alpha : A\rightarrow A$ is an  even linear map, called the twisting map.\\
2) If in addition,  $\alpha\circ\mu_i=\mu_i\circ(\alpha\otimes \alpha)$ $(i=1, \dots, n)$, the color Hom-algebra
 $(A, \mu_1, \dots, \mu_n, \varepsilon, \alpha)$ is said to be multiplicative.\\
3) We call a Hom-$X$ color algebra a color Hom-algebra for which the axioms defining the structure of $X$ are linear combination of the terms
of the form either
$\mu_j\circ(\mu_i\otimes \alpha)$ or $\mu_j\circ(\alpha\otimes\mu_i)$.
 \end{definition}
 \begin{example}
1) A Hom-$X$ color algebra $(A, \mu, \varepsilon, \alpha)$ for which
\begin{eqnarray}
 \mu(\mu(x, y), \alpha(z))-\mu(\alpha(x), \mu(y, z))=0,\label{x1}
\end{eqnarray}
 is called a Hom-associative color algebra.\\
 2) A Hom-$X$ color algebra $(A, \mu, \varepsilon, \alpha)$ for which
\begin{eqnarray}
 \mu(\alpha(x), \mu(y, z))+\varepsilon(y, z)\mu(\mu(x, z), \alpha(y)) -\mu(\mu(x, y), \alpha(z))=0,\label{x2}
\end{eqnarray}
 is called a Hom-Leibniz color algebra.\\
3) A Hom-$X$ color algebra $(A, \mu_1, \mu_2, \varepsilon, \alpha)$ for which \eqref{x1} and \eqref{x2} hold for $\mu_1$
and $\mu_2$ respectively and
\begin{eqnarray}
 \mu_2(\alpha(x), \mu_1(y, z))+\varepsilon(y, z)\mu_2(\mu_1(x, z), \alpha(y)) -\mu_1(\mu_2(x, y), \alpha(z)=0
\end{eqnarray}
 is called a Hom-Leibniz-Poisson color algebra.
 \end{example}
\begin{remark}
The following color Hom-algebras enter in these categories of algebras : Hom-Lie color algebras, Hom-pre-Lie (or left Hom-symmetric) color
algebras, Hom-post-Lie color algebras, Hom-left symmetric color dialgebras, Hom-Poisson color dialgebras, Hom-pre-Poisson color algebras,
Hom-post-Poisson color algebras, Hom-Leibniz-Poisson color algebras, Hom-tridendriform color algebras and so on.
\end{remark}

 \begin{definition}
Let $(A, \mu_1, \dots, \mu_n, \varepsilon, \alpha)$ and $(A', \mu'_1, \dots, \mu'_n, \varepsilon, \alpha')$ be two color Hom-algebras.
An even linear map
$f : A\rightarrow A'$ is said to be a morphism of color Hom-algebras if $f\circ\alpha=\alpha'\circ f$ and
\begin{eqnarray}
 f(\mu_i(x, y))=\mu_i'(f(x), f(y)),\nonumber
\end{eqnarray}
for all $x, y\in\mathcal{H}(A)$ and $i=1, \dots, n$.
 \end{definition}

 \begin{definition}
 Let $(A, \mu_1, \dots, \mu_n, \varepsilon, \alpha)$ be a multiplicative color Hom-algebra and $k\in\mathbb{N}^*$.
\begin{enumerate}
 \item [1)] The $kth$ derived color
Hom-algebra of type $1$ of $A$ is defined by
\begin{eqnarray}
 A_1^k=(A,  \mu^{(k)}_1=\alpha^k\circ\mu_1, \dots, \mu^{(k)}_n=\alpha^k\circ\mu_n, \varepsilon, \alpha^{k+1}).
\end{eqnarray}
\item [2)] The $kth$ derived color Hom-algebra of type $2$ of $A$ is defined by
\begin{eqnarray}
 A_2^k=(A,  \mu^{(2^k-1)}_1=\alpha^{2^k-1}\circ\mu_1, \dots, \mu^{(2^k-1)}_n=\alpha^{2^k-1}\circ\mu_n, \varepsilon, \alpha^{2^k}).
\end{eqnarray}
\end{enumerate}
Note that $A_1^0=A_2^0=(A, \mu_1, \dots, \mu_n, \varepsilon, \alpha)$ and
 $A^1_1=A_2^1=(A, \alpha\circ\mu_1, \dots, \alpha\circ\mu_n, \varepsilon, \alpha^{2})$.
 \end{definition}

\begin{definition}
 A color Hom-algebra $(A, \mu_1, \dots, \mu_n, \varepsilon, \alpha)$ endowed with an even linear map $R : A\rightarrow A$ such that
\begin{eqnarray}
 \mu_i(R(x), R(y))=R\Big(\mu_i(R(x), y)+\mu_i(x, R(y))+\lambda\mu_i(x, y)\Big),\;\; i=1, \dots, n, \label{rb}
\end{eqnarray}
with $\lambda\in\mathbb{K}, x, y\in \mathcal{H}(A)$, is called a Rota-Baxter color Hom-algebra, and $R$ is called a Rota-Baxter operator on $A$.
\end{definition}
The below result allows to get Hom-$X$ color algebras from either an $X$-color algebra on the one hand or another Hom-$X$ color algebra on the
other hand.
\begin{theorem}\label{gftp1}
Let $(A, \mu_1, \dots, \mu_n, \varepsilon, \alpha)$ be a Rota-Baxter Hom-$X$ color algebra and $\beta : A\rightarrow A$ be an endomorphism of $A$.
Then, for any nonnegative integer $n$,
$$A_\beta=(A, \mu_\beta^1=\beta^n\circ\mu_1, \dots, \mu_\beta^n=\beta^n\circ\mu_n, \beta^n\circ\alpha)$$
 is a Rota-Baxter Hom-$X$ color algebra, where $\beta^n=\beta\circ\beta^{n-1}$.

Moreover, suppose that $(A', \mu'_1, \dots, \mu'_n, \varepsilon, \alpha')$ is another Hom-$X$ color algebra
and $\beta' : A'\rightarrow A'$ be an algebra endomorphism.
 If $f : A\rightarrow A'$ is a morphism of Hom-$X$ color algebras that satisfies
 $f\circ\beta=\beta'\circ  f$,
then $f : A_{\beta}\rightarrow A'_{\beta'}$ is also a morphism of Hom-$X$ color algebras.
\end{theorem}
\begin{proof}
The proof of the first part follows from the following facts.

For any $x, y, z\in\mathcal{H}(X), \; 1\leq i, j\leq n$,
\begin{eqnarray}
 \mu_\beta^i(\mu_\beta^j(x, y), (\beta^n\circ\alpha)(z))
&=&\mu_\beta^i(\mu_\beta^j(x, y), \beta^n(\alpha(z)))=\beta^n\mu_i\Big(\beta^n\mu_j(x, y), \beta^n(\alpha(z))\Big) \nonumber\\
&=&(\beta^n\circ\beta^n)\Big(\mu_i(\mu_j(x, y), \alpha(z))\Big)=\beta^{2n}\Big(\mu_i(\mu_j(x, y), \alpha(z))\Big), \nonumber
\end{eqnarray}
and
\begin{eqnarray}
  \mu_\beta^i((\beta^n\circ\alpha)(x), \mu_\beta^j(y, z))
&=& \mu_\beta^i(\beta^n(\alpha(x)), \mu_\beta^j(y, z))
= \beta^n\Big(\mu_i(\beta^n(\alpha(x)), \beta^n(\mu_j(y, z))\Big)\nonumber\\
&=&(\beta^n\circ\beta^n)\Big(\mu_i(\alpha(x), \mu_j(y, z))\Big)=\beta^{2n}\Big(\mu_i(\alpha(x), \mu_j(y, z))\Big).\nonumber
\end{eqnarray}
To prove the Rota-Baxter identity \eqref{rb} for $\mu_\beta^i$, we have
\begin{eqnarray}
 \mu_\beta^i(x, y)
&=&\beta^n(\mu_i(x, y))=\beta^n\Big(\mu_i(R(x), y)+\mu_i(x, R(y))+\lambda\mu_i(x, y)\Big)\nonumber\\
&=&\beta^n\mu_i(R(x), y)+\beta^n\mu_i(x, R(y))+\lambda\beta^n\mu_i(x, y)\nonumber\\
&=&\mu_\beta^i(R(x), y)+\mu_\beta^i(x, R(y))+\lambda\mu_\beta^i(x, y)\nonumber.
\end{eqnarray}

For the second assertion, we have
\begin{eqnarray}
f(\mu_\beta^i(x, y))&=&f(\beta^n(\mu_i(x, y)))=f(\mu_i(\beta^n(x), \beta^n(y))=\mu'_i(f(\beta^n(x)), f(\beta^n(y))\nonumber\\
&=&\mu'_i(\beta'^n(f(x)), \beta'^n(f(y))=\beta'^n(\mu'_i(f(x), f(y))=\mu_\beta'^i(f(x), f(y)) \nonumber.
\end{eqnarray}
This ends the proof.
\end{proof}
 \begin{remark}
 Whenever $\alpha=Id$ in the first part of the previous theorem, we obtain a Hom-$X$ color algebra from an $X$-color algebra.
\end{remark}
\begin{example}\label{ftp1}
 Let $(T, \dashv, \vdash, \cdot, \varepsilon, \alpha)$ be a  Hom-tridendriform color algebra and $\beta : T\rightarrow T$ be an  endomorphism of
$T$. Then, $T_\beta=(T, \dashv_\beta=\beta^n\circ\dashv, \vdash_\beta=\beta^n\circ\vdash, \cdot_\beta=\beta^n\circ\cdot, \beta^n\circ\alpha)$
is a Hom-tridendriform color algebra, for any nonnegative integer $n$.

Moreover, suppose that  $(T', \dashv', \vdash', \cdot', \varepsilon, \alpha')$ is another  Hom-tridendriform color algebra
 and $\beta' : T'\rightarrow T'$ a Hom-tridendriform color algebra  endomorphism. If $f : T\rightarrow T'$ is a morphism of  Hom-tridendriform
 color algebra  that satisfies $f\circ\beta=\beta'\circ  f$, then $ f : T_\beta\rightarrow T'_{\beta'}$
is a morphism of  Hom-tridendriform color algebras.
\end{example}
\begin{proof}
 We shall only prove relation \eqref{t1}, the others being  proved analogously. Then, for any $x, y, z\in\mathcal{H}(T)$,
\begin{eqnarray}
 (x\dashv_\beta y)\dashv_\beta(\beta^n\circ\alpha)(z)
&=&\beta^{n}\Big(\beta^n(x\dashv y)\dashv(\beta^n\circ\alpha)(x)\Big)=\beta^{2n}\Big((x\dashv y)\dashv\alpha(x)\Big)\nonumber\\
&=&\beta^{2n}\Big(\alpha(x)\dashv(y\dashv z+\varepsilon(z, y)y\vdash z+\varepsilon(z, y)y\cdot z\Big)\nonumber\\
&=&(\beta^n\circ\alpha)(x)\dashv_\beta\beta^n\Big(y\dashv z+\varepsilon(z, y)y\vdash z+\varepsilon(z, y)y\cdot z\Big)\nonumber\\
&=&(\beta^n\circ\alpha)(x)\dashv_\beta(y\dashv_\beta z+\varepsilon(z, y)y\vdash_\beta z+\varepsilon(z, y)y\cdot_\beta z)\nonumber.
\end{eqnarray}
For the second assertion, we have
\begin{eqnarray}
f(x\dashv_\beta y)&=&f(\beta^n(x)\dashv\beta^n(y)))=f(\beta^n(x))\dashv' f(\beta^n(y)))\nonumber\\
&=&\beta'^n(f(x))\dashv'\beta'^n(f(y))=f(x)\dashv'_{\beta'} f(y)\nonumber.
\end{eqnarray}
 This completes the proof.
\end{proof}
We have the following series of consequence of Theorem \ref{gftp1}.
\begin{corollary}
Let $(A, \mu_1, \dots, \mu_n, \varepsilon)$ be an $X$-color algebra and $\beta : A\rightarrow A$ be an endomorphism of $A$.
Then $A_\beta= (A, \beta\circ\mu_1, \dots, \beta\circ\mu_n, \varepsilon, \beta)$ is a multiplicative  Hom-$X$ color algebra.
\end{corollary}
\begin{proof}
 Take $\alpha=Id$ in Theorem \ref{gftp1}.
\end{proof}
\begin{corollary}
 Let $(A, \mu_1, \dots, \mu_n, \varepsilon, \alpha)$ be a Hom-$X$ color algebra such that $\alpha$ be invertible.
Then
$ (A, \mu_{\alpha^{-1}}={\alpha^{-1}}\circ\mu_1, \dots, \mu_{\alpha^{-1}}={\alpha^{-1}}\circ\mu_n, \varepsilon)$
is an $X$-color algebra.
\end{corollary}
\begin{proof}
 Take $n=1$ and  $\beta={\alpha^{-1}}$ in Theorem \ref{gftp1}.
\end{proof}
\begin{corollary}
 Let $(A, \mu_1, \dots, \mu_n, \varepsilon, \alpha)$ be a Hom-$X$ color algebra. Then the $k$th derived color Hom-algebra of type $1$ and
 the $k$th derived color Hom-algebra of type $2$ are  Hom-$X$ color algebras.
\end{corollary}
\begin{proof}
 It is sufficient to take $\beta=\alpha$, and  $n={k}$ and $n={2^k-1}$ respectively in Theorem \ref{gftp1}.
\end{proof}

%%%%%%%%%%%%%%%%%%%%%%%%%%%%%%%%%%%%%%%%%%%%%%%%%%%%%

Now we introduce the notion of centroids for Hom-$X$ color algebras. They allow to provide Hom-$X$ color algebras from a given one, and thus they
play here the role of twisting.
 \begin{definition}
Let $(A, \mu_1, \dots, \mu_n, \varepsilon, \alpha)$ be a color Hom-algebra.  An even linear map $\beta : A\rightarrow A$  is said to be
 an element of the centroid if $\beta(\mu_i(x, y))=\mu_i(\beta(x), y)=\mu_i(x, \beta(y))$ for any $x, y\in\mathcal{H}(A)$. \\
The centroid of $A$ is defined by
$$Cent(A)=\Big\{\beta : A\rightarrow A\;\;\mbox{ even linear map} \mid \beta(\mu_i(x, y))=\mu_i(\beta(x), y)=\mu_i(x, \beta(y)),\;\;
 \forall x, y\in\mathcal{H}(A)\Big\}.$$
\end{definition}

\begin{theorem}
 Let $(A, \mu_1, \dots, \mu_n, \varepsilon, R, \alpha)$ be a Rota-Baxter  Hom-$X$ color algebra and $\beta_1, \beta_2 :A\rightarrow A$ be
 a pair of commuting elements of the centroid such that $\beta_i\circ R=R\circ\beta_i, i=1,2$.
\begin{enumerate}
 \item
[1)] Define bilinear maps $\mu^i_{\beta} : A\times A\rightarrow A, (i=1, \dots, n)$, by
$$\mu^i_{\beta}(x, y):=\mu_i((\beta_2\beta_1)(x), y).$$
Then $A_{\beta_1, \beta_2}^1=(A, \mu^1_\beta, \dots, \mu^n_\beta, \varepsilon, R, \beta_2\beta_1\alpha)$ is also a Rota-Baxter Hom-$X$ color algebra.
\item
[2)] Define bilinear maps $\mu^i_{\beta} : A\times A\rightarrow A, (i=1, \dots, n)$, by
$$\mu^i_{\beta}(x, y):=\mu_i(\beta_1(x), \beta_2(y)).$$
Then $A_{\beta_1, \beta_2}^2=(A, \mu^1_\beta, \dots, \mu^n_\beta, \varepsilon, R, \beta_2\beta_1\alpha)$ is also a Rota-Baxter Hom-$X$ color algebra.
\end{enumerate}
\end{theorem}
\begin{proof}
\begin{enumerate}
 \item [1)] For any $x, y\in\mathcal{H}(A)$, we have
 \begin{eqnarray}
\mu^i_\beta(\mu^j_{\beta}(x, y), (\beta_2\beta_1\alpha)(z)) &=& \mu^i_\beta(\mu^j_{\beta}(x, y), \beta_2(\beta_1(\alpha(z)))) \nonumber\\
&=&\mu_i\Big(\beta_2\beta_1\Big(\mu_j(\beta_2\beta_1(x), y)\Big), \beta_2\beta_1(\alpha(z))\Big) \nonumber\\
&=&\beta_2\mu_i\Big(\beta_1\Big(\beta_2\mu_j(\beta_1(x), y)\Big), \beta_2\beta_1(\alpha(z))\Big) \nonumber\\
&=&\beta_2^2\mu_i\Big(\beta_1\Big(\beta_2\mu_j(\beta_1(x), y)\Big), \beta_1(\alpha(z))\Big)\nonumber\\
&=&\beta_2^2\beta_1\mu_i\Big(\beta_2\mu_j(\beta_1(x), y)), \beta_1(\alpha(z))\Big) \nonumber\\
&=&\beta_2^3\beta_1\mu_i\Big(\beta_1\mu_j(x, y)), \beta_1(\alpha(z))\Big)\nonumber\\
&=&\beta_2^3\beta_1^2\mu_i\Big(\mu_j(x, y)), \beta_1(\alpha(z))\Big) \nonumber\\
&=&\beta_2^3\beta_1^3\mu_i(\mu_j(x, y), \alpha(z))\nonumber.
 \end{eqnarray}
We have a similar proof for  $\mu^i_\beta((\beta_2\beta_1\alpha)(x), \mu^j_{\beta}(y, z))$.\\
Next,
 \begin{eqnarray}
\mu^i_\beta(R(x), R(y))
&=& \mu_i(\beta_2\beta_1(R(x)), R(y))=\beta_2\beta_1\mu_i(R(x), R(y))\nonumber\\
&=&\beta_2\beta_1R\Big(\mu_i(R(x), y)+\mu_i(x, R(y))+\lambda\mu_i(x, y)\Big)\nonumber\\
&=& R\Big(\beta_2\beta_1\mu_i(R(x), y)+\beta_2\beta_1\mu_i(x, R(y))+\lambda\beta_2\beta_1\mu_i(x, y)\Big)\nonumber\\
&=& R\Big(\mu_i(\beta_2\beta_1R(x), y)+\mu_i(\beta_2\beta_1(x), R(y))+\lambda\mu_i(\beta_2\beta_1(x), y)\Big)\nonumber\\
&=& R\Big(\mu^i_\beta(R(x), y)+\mu^i_\beta(x, R(y))+\lambda\mu^i_\beta(x, y)\Big)\nonumber.
 \end{eqnarray}
\item [2)] For any $x, y\in\mathcal{H}(A)$, we have
 \begin{eqnarray}
  \mu^i_\beta(\mu^j_{\beta}(x, y), (\beta_2\beta_1\alpha)(z))
&=&\mu^i_\beta(\mu^j_{\beta}(x, y), \beta_2(\beta_1(\alpha(z)))) \nonumber\\
&=&\mu_i\Big(\beta_1\Big(\mu_j(\beta_1(x), \beta_2(y))\Big), \beta_2(\beta_2\beta_1(\alpha(z)))\Big)  \nonumber\\
&=&\beta_1\beta_2\mu_i\Big(\beta_1\beta_2\Big(\mu_j(x, y)\Big), \beta_2\beta_1(\alpha(z))\Big) \nonumber\\
&=&\beta_1^2\beta_2^2\mu_i\Big(\mu_j(x, y), \beta_2\beta_1(\alpha(z))\Big),  \nonumber\\
&=&\beta_1^3\beta_2^3\mu_i(\mu_j(x, y), \alpha(z)) , \nonumber
 \end{eqnarray}
Similarly, we can prove that $\mu^i_\alpha((\beta_2\beta_1\alpha)(x), \mu^j_{\alpha}(y, z))=\beta_1^3\beta_2^3\mu_i(\alpha(x), \mu_j(y, z))$.\\
Next,
 \begin{eqnarray}
\mu^i_\beta(R(x), R(y))
&=& \mu_i(\beta_1(R(x)), \beta_2R(y))=\beta_2\beta_1\mu_i(R(x), R(y))\nonumber\\
&=&\beta_2\beta_1R\Big(\mu_i(R(x), y)+\mu_i(x, R(y))+\lambda\mu_i(x, y)\Big)\nonumber\\
&=& R\Big(\beta_2\beta_1\mu_i(R(x), y)+\beta_2\beta_1\mu_i(x, R(y))+\lambda\beta_2\beta_1\mu_i(x, y)\Big)\nonumber\\
&=& R\Big(\mu_i(\beta_1R(x), \beta_2(y))+\mu_i(\beta_1x, \beta_2(R(y)))+\lambda\mu_i(\beta_1(x), \beta_2(y))\Big)\nonumber\\
&=& R\Big(\mu^i_\beta(R(x), y)+\mu^i_\beta(x, R(y))+\lambda\mu^i_\beta(x, y)\Big)\nonumber.
 \end{eqnarray}
\end{enumerate}
This finishes the proof.
\end{proof}
Let us observe that $A_{\beta_1, Id}^1=A_{\beta_1, Id}^2$.
\begin{corollary}\label{pro}
 Let $(A, \mu_1, \dots, \mu_n, \varepsilon, R, \alpha)$ be a Rota-Baxter  Hom-$X$ color algebra and $\beta\in Cent(A)$.
Define bilinear maps $\mu^i_{\beta} : A\times A\rightarrow A, (i=1, \dots, n)$, by
$$\mu^i_{\beta}(x, y):=\mu_i(\beta(x), y).$$
Then $(A, \mu^1_\beta, \dots, \mu^n_\beta, \varepsilon, R, \beta\circ\alpha)$ is also a Rota-Baxter Hom-$X$ color algebra.
\end{corollary}

Remark that $(A, \mu^1_{\beta}, \dots, \mu^n_{\beta}, \varepsilon, \beta)$ is also a Hom-$X$ color algebra whenever $\alpha$ is the identity map
in corollary \ref{pro}.

\section*{Acknowlegments} Ibrahima Bakayoko is grateful to the research environment in Mathematics and
Applied Mathematics MAM, the Division of Applied Mathematics
of the School of Education, Culture and Communication
at M{\"a}lardalen University for hospitality and an excellent and inspiring environment for research and research education and cooperation in Mathematics during his visit, in the framework of research and research education capacity and cooperation development programs in Mathematics between Sweden and countries in Africa supported by Swedish International Development Agency (Sida) and International Program in Mathematical Sciences (IPMS).

\end{document}